\documentclass[11pt]{article}
\usepackage{amsmath}
\usepackage{amssymb}
\usepackage{amsthm}
\usepackage{lineno}
\usepackage{graphicx}
\usepackage{caption}
\usepackage{subcaption}
\usepackage{blkarray}
\usepackage{enumerate}
\usepackage{enumitem}
\usepackage[font=small,skip=0pt]{caption}
\usepackage{xcolor}
\usepackage{verbatim}
\usepackage{geometry}
\usepackage{enumerate}
\usepackage{wrapfig}
\usepackage{color,soul}
\usepackage{float}
\usepackage{lineno}
\usepackage[stable]{footmisc}
\usepackage{lineno}
\usepackage{blindtext}
\usepackage{calc}
\usepackage{upgreek}
\usepackage{amsmath}
\usepackage{amssymb}
\usepackage{amsthm}
\usepackage{lineno}
\usepackage{graphicx}
\usepackage{caption}
\usepackage{subcaption}
\usepackage[font=small,skip=0pt]{caption}
\usepackage{xcolor}
\usepackage{verbatim}
\usepackage{geometry}
\usepackage[colorinlistoftodos]{todonotes}
\usepackage{wrapfig}
\usepackage{color,soul}
\usepackage{float}
\usepackage{url}
\usepackage[colorinlistoftodos]{todonotes}
\usepackage{wrapfig}
\usepackage{color,soul}
\usepackage{float}
\usepackage{multirow}

\numberwithin{equation}{section}
\newtheorem{Theorem}{Theorem}
\newtheorem{Corollary}{Corollary}

\newtheorem*{Note*}{Note}
\newtheorem{Proposition}{Proposition}
\newtheorem{Remark}{Remark}
\newtheorem{Lemma}{Lemma}

\newtheorem*{Recall*}{Recall}
\newcommand{\geqs}{\geqslant}
\newcommand{\leqs}{\leqslant}
\graphicspath{{Figures/}}
\usepackage{color}

\definecolor{ao(english)}{rgb}{0.0, 0.5, 0.0}

\usepackage[utf8]{inputenc}
\usepackage[titles]{tocloft}

\title{Gerber-Shiu Theory for Discrete Risk Processes in a Regime Switching Environment}
\author{
        Zbigniew Palmowski$^{a,}$\footnote{Corresponding author e-mail address: zbigniew.palmowski@gmail.com. This work is partially supported by Polish National Science Centre Grant
No. 2021/41/B/HS4/00599 (2022-2026).}\,, \,  Lewis Ramsden$^{b,}$\footnote{Second author e-mail address: lewis.ramsden@york.ac.uk. This work is partially supported by London Mathematical Society Grant Ref.\,41911.}\, and   Apostolos D. Papaioannou$^{c,}$\footnote{Third author  e-mail address: papaion@liverpool.ac.uk}     \\
      \\ $^a$Department of Applied Mathematics\\
       Wroc\l{}aw University of Science and Technology \\
       Wroc\l{}aw, Poland; \\
        $^b$ The School for Business and Society, University of York\\
       York, Yorkshire, YO10 5DD, United Kingdom;\\
       $^c$Institute for Financial and Actuarial Mathematics \\
               Department of Mathematical Sciences\\
        University of Liverpool, L69 7ZL,  United Kingdom
       }

\begin{document}
\maketitle

\begin{abstract}
\noindent In this paper we develop the Gerber-Shiu theory for the classic and dual discrete risk processes in a Markovian (regime switching) environment. In particular, by expressing the Gerber-Shiu function in terms of potential measures of an upward (downward) skip-free discrete-time and discrete-space Markov Additive Process (MAP), we derive closed form expressions for the Gerber-Shiu function in terms of the so-called (discrete) $\bold{W}_{v}$ and $\bold{Z}_{v}$ scale matrices, which were introduced in \cite{ourpaper}. We show that the discrete scale matrices allow for a unified approach for identifying the Gerber-Shiu function as well as the value function of the associated constant dividend barrier problems.
\end{abstract}

\noindent {\bf Keywords:} Gerber-Shiu; Discrete-Time; Dual Risk Process; Markov Additive Process; Scale Matrices; Exit Problems; Dividends; Markov-Modulation

\section{Introduction}
\label{Intro}

Gerber-Shiu theory lies at the heart of modern risk and ruin theory as a unifying method of analysis for a number of popular risk measures via the so-called expected discounted penalty or Gerber-Shiu (G-S) function. Shortly after its introduction in the seminal paper \cite{GerberShiu}, the G-S function received a great deal of attention for a variety of risk models and has led to a huge library of literature, see \cite{Cheng2000}, \cite{Gerber2005} and \cite{li2005general}, to name only a select few. Following these initial developments, G-S theory attracted further attention, and its versatility has been explored for more exotic risk models, including for e.g.\, investment income (\cite{Paulsen2008-lj}) and dividend barriers (\cite{Lin2003}, \cite{Yang2008-yn}), among others. In fact, over the years the initial construction of the G-S function has been adapted to include further risk related quantities, e.g., the minima prior to ruin \cite{Biffis2010} and the number of claims until ruin \cite{Frostig2012}, without altering the tractability of its results and G-S theory has now become an umbrella term for a number of other risk related quantities including the expected discounted dividends and accumulated capital injections until ruin, to name a few. For a general overview of the G-S literature, the readers are directed to \cite{asmussen2010ruin}, \cite{kyprianou2013gerber} and references therein.

One particular class of continuous-time risk models for which G-S theory has been developed in more recent years are L\'evy insurance risk models, see for example \cite{Avrametal.2015}, \cite{Garrido2006} and \cite{kyprianou2013gerber}. For this general class of processes, for which a number of classical risk models can be seen as special cases: the compound Poisson, diffusion and Poisson jump-diffusion models, the G-S function can be expressed in terms of the so-called $W^{(q)}$ and $Z^{(q)}$ scale functions, which provide an over-arching representation for many previously derived expressions. In fact, this unifying approach has been shown to hold in the even larger class of Markov additive risk processes (MAP), for which an external Markov process influencing the underlying distribution of the risk process is considered (Markov-modulation), via the existence of so-called scale matrices,  see \cite{feng2014potential}.

The discrete-time analogue of the G-S theory (with and without dividend barriers) has also received some attention over the years but to a much lesser extent to that of the continuous-time setting and only for specific risk models. For example, ruin probabilities and other risk related quantities for the compound binomial risk model can be found in \cite{Chen2014-hf} and \cite{Willmot1993DR}, whilst  \cite{Wu2009} derive a recursive expression for the G-S function itself in a  discrete-time renewal risk model with arbitrary inter-arrival claim times. Markov-modulation has also been considered in the discrete-time setting through so-called semi-Markov models introduced in \cite{reinhard2000probability}, \cite{reinhard2001distribution} and \cite{reinhard2002severity}, where special cases of the G-S function were considered and later generalised in \cite{chen2014survival} who derive recursive formulae for the survival probabilities under weaker conditions. More recently, \cite{Chen2014-hf} derive a closed form expression for the expected discounted dividends for the semi-Markov risk model in terms of an auxiliary function satisfying a recursive expression, whilst \cite{kim2022parisian} obtain a matrix expression for the G-S function for the dual semi-Markov risk model which is then used to determine Parisian type ruin probabilities. Although each of the papers mentioned above derive individual results for the G-S function and other related quantities, to the best of the authors' knowledge, there does not exist an over-arching, unifying set of expressions for G-S theory in discrete-time like those of the scale functions/matrices for MAPs in continuous-time.

The aim of this paper is to derive such unifying expressions for the G-S function and the expected discounted dividends until ruin for a general discrete Markov additive-type risk model (Markov-modulated random walk) and its `dual' counterpart in terms of discrete scale matrices. This is done by first deriving results from potential theory which provides a connection between the G-S function and the fluctuation theory results for a Markov additive chain derived in \cite{ourpaper}.

The rest of this paper is organised as follows: In Section 2, we introduce a general Markov additive-type risk model, its dual counterpart and the corresponding G-S functions. In Section 3, we present an overview of the results from fluctuation theory of Markov additive chains that are utilised in the subsequent sections to derive expressions for the G-S function and expected discounted dividends until ruin. In Section 4, we derive results for the associated potential measure of the risk process, which allows us to find closed form expressions for the G-S functions. Finally, within Section 4, we introduce the value function for the expected discounted dividends until ruin and use the previous theory to derive closed form expressions for these quantities for, both, the regular and the dual Markov additive risk models.

\section{Risk Models and Gerber-Shiu Function}
\label{sec:models}
\noindent Let us define a discrete-time risk process, denoted $\{U_n\}_{n \in \mathbb{N}}$, which models the reserve of an insurer at time $n \in \mathbb{N}$, by
\begin{linenomath*}
\begin{equation}
\label{eqCDTRM1}
U_n=u+n-\sum_{i=1}^n Z_i,
\end{equation}
\end{linenomath*}
where $u \in \mathbb{N}$ represents the insurers (integer) initial reserve, premium is received at a unit rate per period of time an $\{Z_k\}_{k \in \mathbb{N}^+}$ is a sequence of integer claim sizes describing the claim size at period $k \in \mathbb{N}^+$. This simple model is known within the literature as the compound binomial risk model and was first introduced as a discrete counter part to the continuous-time Poisson risk model in \cite{Gerber1988}.

Within a Markovian environment, the above risk process is further influenced by an underlying discrete-time homogeneous Markov chain, denoted by $\{J_n\}_{n \in \mathbb{N}}$ with finite state space $E = \left\{ 1, 2, \ldots, N \right\}$, which describes the \textit{phase} of the risk process at period  $n \in \mathbb{N}$ having transition probability matrix $\bold{P}$, with $(i,j)$-th element
\begin{linenomath*}\begin{equation}\label{TransProb}
p_{ij} := \mathbb{P} \left( J_1 = j | J_{0} =i \right)
\end{equation}\end{linenomath*}
and influences the risk process, $\{U_n\}_{n \in \mathbb{N}}$, through the claim size distributions. That is, we assume that the random non-negative claim amounts, namely $\{Z_k\}_{k \in \mathbb{N}^+}$, are \textit{conditionally} independent and identically distributed (i.i.d.)\,random variables, given $\{J_{k-1}=i\}$, with distribution described by the probability mass matrix $\bold{\Lambda}(\cdot)$, with $(i,j)$-th element
\begin{linenomath*}\begin{equation}\label{ClaimMat}
\lambda_{ij}(m) :=\mathbb{P}(Z_1=m, J_1 = j | J_0 = i), \quad \text{for} \quad m=0,1,2, \ldots,
\end{equation}\end{linenomath*}
and finite means $\mathbb{E}(Z_1\mathbb{I}_{(J_1 = j)}|J_0 = i)<\infty$ for all $i,j \in E$. We point out, due to its importance in the following, that the claim amounts $\{Z_k\}_{k \in \mathbb{N}^+}$ have a mass point at zero with probability $\lambda_{ij}(0)>0$ for some, $i,j \in E$.

Due to the models phase dependencies described above, it will be convenient in this paper to introduce a probability measure matrix $\mathbb{P}\left( \cdot, J_n\right)$ and corresponding expectation operator matrix $\mathbb{E}\left(\cdot\,; J_n\right)$, having $(i,j)$-th elements $\mathbb{P}(\cdot, J_n = j | J_0 = i)$ and $\mathbb{E}\left( \cdot \mathbb{I}_{(J_n = j)} |J_0 = i\right)$, respectively, for $i,j \in E$.

Of principle interest within G-S theory are the distributional properties related to the \textit{time of ruin} which are obtained via the so-called G-S function. For the discrete-time Markovian risk model given in Eq.\,\eqref{eqCDTRM1}, we define the time of ruin by
\begin{linenomath*}\begin{equation*}
\label{eqTimeRuin}
\tau_0=\inf \{n\in \mathbb{N}: U_n  \leqs 0\},
\end{equation*}\end{linenomath*}
with $\tau_0=\infty$ if $U_n>0$ for all $n\in \mathbb{N}$ and the G-S function, denoted $\phi(u)$, by
\begin{linenomath*}\begin{equation}
\phi(u) =  \boldsymbol{\alpha}^\top \bold{\Phi}(u) \boldsymbol{e},
\end{equation}\end{linenomath*}
where $\boldsymbol{\alpha}^\top= \left( \alpha_1, \ldots, \alpha_N\right)$ with $\alpha_i = \mathbb{P}\left( J_0 = i \right)$, for $i \in E$, denotes the initial distribution vector of $\{J_n\}_{n \in \mathbb{N}}$, $\boldsymbol{e}$ is the column vector of units and the $N$-dimensional square matrix $\bold{\Phi}(u)$ has $(i,j)$-th element

\begin{linenomath*}\begin{equation}\label{GSfunction}
\phi_{ij}(u):=
\mathbb{E}\left[v^{\tau_0}\omega (U_{\,{\tau_0-1}}, |U_{\,\tau_0}|)\mathbb{I}_{(\tau_0<\infty, J_\tau=j)}|J_0=i, U_0 = u\right].
\end{equation}\end{linenomath*}
The function $\omega: \mathbb{N}^+\times \mathbb{N} \rightarrow \mathbb{R}^+$ (non-negative real line) denotes a penalty function and $v\in (0,1]$ is a discounting factor. For the case where $v=1$ and $\omega(\cdot,\cdot)\equiv 1$, the G-S functions, $\phi_{ij}(u)$, reduces to the conditional infinite-time ruin probabilities
\begin{equation}
\label{eq:ruin}
\psi_{ij}(u) = \mathbb{P}\left(\tau_0 < \infty, J_{\tau_0} = j | J_0=i, U_0 = u \right).
\end{equation}

\begin{Remark}
 The definition of the ruin time $\tau_0$ given above is similar to \cite{Gerber1988}, whilst other authors define the ruin time in discrete models as the first time the reserve takes strictly negative values (see for example \cite{Willmot1993DR}).
\end{Remark}

\begin{Remark}
It is worth noting that the G-S function defined above could be further generalised  by considering a Markov dependent penalty function, $\omega_{ij}(\cdot)$. As the aim of this paper is to present a unifying theory which can be used to obtain known results from the G-S literature, this generalisation is not included explicitly here to allow for the comparison to previous findings. However, the reader should keep in mind that the following results are implicitly more general than they may appear.
\end{Remark}

\noindent In addition to the risk process given in Eq.\,\eqref{eqCDTRM1}, we are also interested in the distribution of the associated `dual' risk process within a Markovian regime-switching environment. The dual risk process, denoted $\{\widehat{U}_n\}_{n \in \mathbb{N}}$, represents a process with (deterministic) unit losses per period and random (integer) gains. As such, the dynamics (jumps) of the dual risk process are equivalent in distribution to those of the reflection of a `regular' process as defined in Eq.\,\eqref{eqCDTRM1}, i.e.

\begin{linenomath*}\begin{equation}\label{eqDual}
\{\Delta\widehat{U}_n \}_{n \in \mathbb{N}}\overset{d}{=} \{-\Delta U_n \}_{n \in \mathbb{N}},
\end{equation}\end{linenomath*}
where $\Delta U_n = U_n - U_{n-1}$. Throughout this paper, the \textit{dual} counterparts of risk processes and associated risk measures will be denoted with the hat symbol,  $\widehat{\cdot}$\,. In particular, the conditional G-S functions for the dual risk process within initial reserve $u \in \mathbb{N}$ are defined by
\begin{linenomath*}\begin{equation}\label{GSfunctiondual}
\widehat{\phi}_{ij}(u):= \mathbb{E}\left[v^{\widehat{\tau}_0}\omega(\widehat{U}_{\, \widehat{\tau_0}-1}, |\widehat{U}_{\,\widehat{\tau}_0}|)\mathbb{I}_{ \left(\widehat{\tau}_0<\infty, J_{\widehat{\tau}_0}=j\right)}|J_0=i, \widehat{U}_0 = u \right],
\end{equation}\end{linenomath*}
where
\begin{linenomath*}\begin{equation*}
\label{eqTimeRuin}
\widehat{\tau}_0=\inf \{n\in \mathbb{N}: \widehat{U}_n \leqs 0\},
\end{equation*}\end{linenomath*}
and the unconditional G-S function is given by
\begin{linenomath*}\begin{equation} \label{DGS}
\widehat{\phi}(u) =  \boldsymbol{\alpha}^\top \widehat{\bold{\Phi}}(u) \boldsymbol{e},
\end{equation}\end{linenomath*}
where $\widehat{\bold{\Phi}}(u)$ is a square matrix with elements $\widehat{\phi}_{ij}(u)$ for $i,j \in E$, as defined in Eq.\,\eqref{GSfunctiondual}. Note that, for the dual risk process, the `surplus' experiences at most a unit decrease per unit of time and thus, it follows that $\widehat{U}_{\, \widehat{\tau}_0 - 1} = 1$ and $\widehat{U}_{\, \widehat{\tau}_0} = 0$ a.s.. In this case, the G-S measure is usually reduced to the discounted or transform of the time to ruin, i.e.,

\begin{linenomath*}\begin{equation}
\label{GSfunctiondual1}
\widehat{\phi}_{ij}(u)= \mathbb{E}\left[v^{\widehat{\tau}_0}\mathbb{I}_{(\widehat{\tau}_0<\infty, J_{\widehat{\tau}_0}=j)}|J_0=i, \widehat{U}_0=u\right],
\end{equation}\end{linenomath*}
where the corresponding ruin probabilities, denoted $\widehat{\psi}_{ij}(u)$, can be recovered by simply setting $v = 1$ for all $i,j \in E$.


The key observation in this paper is that the risk process $\{U_n\}_{n \in \mathbb{N}}$ defined in Eq.\,\eqref{eqCDTRM1} and its dual counterpart, are in fact both of the form of a discrete-time and discrete space (lattice) MAP, also known as a \textit{Markov Additive Chain} (MAC). As such, we can exploit the fluctuation theory developed in \cite{ourpaper} for this general class of processes and express the G-S function(s) in terms of so-called \textit{scale matrices}. Hence, in Section \ref{MACs}, we will introduce the theory for MACs, along with the key results and observation derived in \cite{ourpaper}, which will later be adapted to the insurance risk set-up as discussed above. In Section \ref{sec:results}, we derive semi-explicit results for the G-S functions defined above as well as results for the associated constant dividend barrier strategies of both the classic and dual risk models.

\section{Markov Additive Chains}
\label{MACs}

Consider a bivariate discrete-time Markov chain $(X,J) = \{(X_n, J_n)\}_{n \in \mathbb{N}}$ on the product space $\mathbb{Z}\times E$, where $X_n \in \mathbb{Z}$ describes the \textit{level process}, whilst $J_n \in E$ is an underlying Markov chain as defined in Section \ref{sec:models}, known as the \textit{phase process} which affects the dynamics of $\{X_n\}_{n \in \mathbb{N}}$. It is assumed throughout that the underlying Markov chain $\{J_n\}_{n \in \mathbb{N}}$ is irreducible and positive recurrent, such that its stationary distribution $\boldsymbol{\pi}^\top = \left( \pi_1, \ldots, \pi_N\right)$ exists and is unique. The process $(X,J)$ is known as a MAC if it satisfies the so-called Markov additive property. That is, given that $\{ J_T = i \}$, for any $T \in \mathbb{N}$ and phase $i \in E$, the Markov chain $\{ (X_{T+n} - X_T, J_{T+n}) \}_{n \in \mathbb{N}}$ is independent of $\mathcal{F}_T$ (the natural filtration to which the bivariate process $(X, J)$ is adapted) and
\begin{linenomath*}\begin{equation*}
\{ (X_{T+n} - X_T, J_{T+n}) \} \overset{d}{=} \{ (X_n - X_0, J_n) \},
\end{equation*}\end{linenomath*}
given $\{J_0 = i\}$. Note that a consequence of the Markov additive property is that the level process $\{X_n\}_{n \in \mathbb{N}}$ is translation invariant on the lattice.

Based on this property, it is easy to see that any MAC is equivalent to a general Markov-modulated random walk where the level process $\{X_n\}_{n \in \mathbb{N}}$ has representation
\begin{linenomath*}\begin{equation}\label{eq:X}
X_n = x + Y_1 + Y_2 + \cdots + Y_n,
\end{equation}\end{linenomath*}
where $X_0=x$, and $\{Y_k\}_{k \in \mathbb{N}^+}$ is a sequence of conditionally i.i.d.\,random variables whose distributions depend on the phase process $\{J_n\}_{n \in \mathbb{N}}$ and described by the joint probability matrix $\bold{A}_m$, for $m \in \mathbb{Z}$, having $(i,j)$-th element
\begin{equation}
a_{ij}(m):= \mathbb{P}\left(Y_1 = m, J_1=j | J_0 =i \right).
\end{equation}

\noindent Although the above definition holds for a general MAC with jumps in either direction, in this paper we are primarily concerned with so-called upward skip-free or `spectrally negative' MACs. That is, we only consider MACs that can `drift' upwards by a maximum of one per unit time and can experience negative jumps only, i.e. $\bold{A}_m = \bold{0}$ (zero matrix), $\forall m \geqslant 2$.

\begin{Remark}
Note that for dual risk process which `drifts' downwards by one per unit time and has only positive jumps, i.e. $\bold{\widehat{A}}_m = \bold{0}$, $\forall m \leqslant -2$, corresponds to a`spectrally positive' MAC. However, by taking advantage of the reflective relationship between the two processes as stated in Eq.\eqref{eqDual}, the results for the dual risk process can, in fact, be expressed in terms of corresponding results for the regular risk process (see below for more details).
\end{Remark}

\noindent It is well known that random walks can be fully characterized by their probability generating functions (p.g.f.) due to their uniqueness and play an important role in their distributional properties. It turns out that this is also true for MACs where the p.g.f.\,is given in matrix form and defined in the following proposition which was proved in \cite{ourpaper}. In the following, we use the notation $\mathbb{E}_x(\cdot) := \mathbb{E}(\cdot | X_0 = x)$ with $\mathbb{E}(\cdot) \equiv \mathbb{E}_0(\cdot)$, where similar notation is employed for the associated probability measures.

\begin{Proposition}
\label{PropPGF2}
For $X_0 = 0$, the p.g.f.\,of the level process $\{X_n\}_{n \geqs 0}$, is given by
\begin{linenomath*}\begin{equation*}
 \mathbb{E}\left( z^{-X_n}\, ; J_n \right) = \bigl(\bold{F}(z)\bigr)^n,
\end{equation*}\end{linenomath*}
where
\begin{linenomath*}\begin{linenomath*}\begin{equation}
\label{eqFMat}
\bold{F}(z) := \mathbb{E}\left( z^{-X_1}\, ; J_1 \right) = \sum_{m=-1}^\infty z^{m}\bold{A}_{-m}.
\end{equation}\end{linenomath*}\end{linenomath*}
In particular, for $z=1$, we have
\begin{linenomath*}\begin{equation*}
\bold{F}(1) = \bold{P} = \sum_{m=-1}^\infty \bold{A}_{-m} .
\end{equation*}\end{linenomath*}
\end{Proposition}

	\noindent The primary concern of this paper is to derive exit problems for the above MAC for different levels or strips. An important property influencing these quantities is the asymptotic drift of the level process $\{X_n\}_{n \in \mathbb{N}}$. It was shown in \cite{AsmussenApplied} that this depends solely on the Perron-Frobenius eigenvalue of the matrix $\bold{F}(z)$, denoted $\kappa(z)$, such that $X_n \rightarrow +(-) \infty$ if and only if $\kappa'(1) <(>)\, 0$, where
	\begin{linenomath*}\begin{equation}\label{eq:drift}
			\kappa'(1) = -\mathbb{E}^{\boldsymbol{\pi}}\left(X_1\right)  =   \boldsymbol{\pi}^\top \sum_{m = -1}^\infty m\, \bold{A}_{-m}\bold{e},
	\end{equation}\end{linenomath*}
	with $\mathbb{E}^{\boldsymbol{\pi}}\left(\cdot \right)$ denoting the expectation operator under the assumption that $J_0$ has stationary initial probability (see \cite{AsmussenApplied} and \cite{ourpaper} for more details).

\begin{Remark}
	These conditions correspond to the so-called net-profit condition often implemented in the risk theory literature to ensure that ruin does not occur a.s., and will be implemented as and where necessary in Section \ref{sec:results}.
\end{Remark}

\noindent For the remainder of this paper, it will be assumed that the matrix $\bold{A}_1$ is non-singular and thus, its inverse $\bold{A}_1^{-1}$ exists. Although this is a somewhat restrictive assumption, it is necessary to present the following results in a consistent way which align with the existing literature. However, \cite{Ivanovs2019} and consequently \cite{ourpaper}, show that the following general results still hold for arbitrary $\bold{A}_1$ but at the cost of familiar representation and comparability to existing results.

\subsection{Occupation Times}

Occupation times describe the number of periods (time) that the MAC $(X,J)$ spends in any given state and forms another fundamental quantity within its analysis.

Let $\bold{L}_v(x, n)$ denote the \textit{occupation mass matrix} describing the discounted time the process $\{(X_n, J_n)\}_{n \in \mathbb{N}}$ spends in state $(x, j)\in \mathbb{Z} \times E$ - up to and including time $n \in \mathbb{N}$ - with $(i,j)$-th element
\begin{linenomath*}\begin{equation}
\label{eqOccDenMat}
\bigl(\bold{L}_v(x, n)\bigr)_{ij} = \mathbb{E}\left(\sum_{k=0}^n v^k \mathbb{I}_{(X_k = x, J_k=j)} \bigg| J_0 = i\right).
\end{equation}\end{linenomath*}

\noindent Due to the strong Markov and Markov additive properties of the MAC, occupation times have proven vital tools for obtaining a variety of results regarding exit problems and other distributional quantities of the MAC and MAPs in general (see \cite{ivanovs2012occupation} and \cite{ivanovs2014potential} among others).  In the following section, we will introduce the concept of discrete scale matrices, their relationship to occupation times and recall some of the main results found in \cite{ourpaper}. One such result worth stating at this point, is how the $z$-transform of the above occupation mass matrix can be expressed in terms of the fundamental p.g.f. $\bold{F}(z)$ of the MAC and is given in the following theorem, which was proved in \cite{ourpaper}.

\begin{Theorem} \label{thm:Lv}
For all $v, z\in (0,1]$ such that $\bold{I} - v\bold{F}(z)$ is non-singular, it follows that
\begin{equation} \label{eq:Lv}
	\sum_{x\in \mathbb{Z}} z^{-x} \bold{L}_v(x, \infty) = \left( \bold{I} - v\bold{F}(z)\right)^{-1}.
\end{equation}
\end{Theorem}

\subsection{Exit Problems and Scale Matrices} \label{sec:ExitProb}
Let us define $\tau_y^{\pm}$, to be the first time the level process $\{X_n\}_{n \in \mathbb{N}}$ up(down)-crosses the level $y \in \mathbb{Z}$, such that
\begin{linenomath*}\begin{equation}
\label{crossingtime}
\tau_y^{+} = \inf \{ n \geqs 0: X_n \geqs y\} \quad \text{and} \quad \tau_y^{-} = \inf \{ n \geqs 0: X_n \leqs y\},
\end{equation}\end{linenomath*}
and the so-called \textit{hitting} times as
\begin{equation} \label{eq:hitting}
\tau^{\{y\}} = \inf\{n \in \mathbb{N}: X_n = y\}.
\end{equation}
Recall that in this paper we are concerned only with `spectrally negative' MACs for which the level process can increase at most one per unit time. A consequence of this is that the upward crossing time, $\tau^+_y$, for $y \geqs  x=X_0$, is equivalent to the hitting time $\tau^{\{y\}}$ and we have $X_{\tau^+_y}  = X_{\tau^{\{y\}}} =  y$. Note that this is not necessarily true for the downward crossing time due to the presence of downward jumps. It turns out that the upward skip-free property of spectrally negative MACs results in $\{J_{\tau^+_y}\}_{y \geqs x}$ being a homogenous Markov chain and, consequently, the so-called \textit{first passage process}, namely $\{(\tau_y^+, J_{\tau^+_y})\}_{y \geqs x}$, being itself a MAC. With this in mind, let us define $\bold{G}_v$ to be the transform of the first hitting time of the upper level $y = 1$, given $X_0 = 0$, such that
\begin{linenomath*}\begin{equation}
\label{GvMatrixOrig}
\mathbb{E}\bigl( {v}^{\tau^+_1} ; J_{\tau^+_1} \bigr) = \bold{G}_{v},
\end{equation}\end{linenomath*}
with $\bold{G}_1 \equiv \bold{G}$ denoting the one-step transition probability matrix of $\{J_{\tau_y^+}\}_{y \geqs x}$. This is known as one of the \textit{fundamental matrices} of MACs, each of which play a vital role in the fluctuation theory (see \cite{ivanovs2012occupation} for the corresponding matrices in the continuous setting).
With this in mind, we have the following theorem which concerns the distribution of the one-sided upward exit times of the MAC, which are given explicitly in terms of the fundamental matrix $\bold{G}_v$ and was proved in \cite{ourpaper}.

\begin{Theorem}[One-sided upward]
	 \label{ThmOneUp}
For $X_0=x$, $v \in (0,1]$ and $a \geqs x$, the transform of the upward crossing/hitting time $\tau^+_a$ satisfies
\begin{linenomath*}\begin{equation}
\label{GvMatrix}
\mathbb{E}_x\bigl( {v}^{\tau^+_a} ; J_{\tau^+_a} \bigr) = \bold{G}_{v}^{a-x},
\end{equation}\end{linenomath*}
where the matrix $\bold{G}_{v}$ is the right solution of the matrix equation $\bold{F}(\cdot) = {v}^{-1}\bold{I}$.
\end{Theorem}

\begin{Remark} \label{rem:numsol}
The matrix $\bold{G}_v$, as the right solution of the above matrix equation, can only be found explicitly in a few special cases. However, there exists a number of numerical algorithms that can be employed to obtain approximations. For a detailed survey of such algorithms, see \cite{bini2005numerical} and references therein. The left solution of this matrix equation is also of importance to the analysis of exit problems and is associated with the time-reversed counterpart of $\bold{G}_v$ (see Section \ref{sec:time-rev}). Moreover, the matrix $\bold{G}_v$ can be shown to be invertible as long as $\bold{A}_1$ is invertible (non-singular) (see Remark 5 of \cite{ourpaper} for details).
\end{Remark}

\noindent For the so-called two-sided exit problem, we are interested in exiting from the (fixed) `strip' $[0,a]$, with $a > 0$. More formally, we are interested in the events $\{ \tau_a^+ < \tau_0^-\}$ and $\{ \tau_a^+ > \tau_0^-\}$, which correspond to the upward and downward exits from the strip, respectively, with $X_0 = x \in [0,a]$. In a similar way to the one-sided upward exit given above, the following theorems from \cite{ourpaper} show that the two-sided exit problems (upward and downward) can be expressed in terms of two other fundamental matrices, known as the $\bold{W}_v(\cdot)$ and $\bold{Z}_v(\cdot)$ scale matrices.

In order to accurately describe the domain for which the following results exist, we will define $\Gamma(\bold{M})$ to be the set of all eigenvalues for the matrix $\bold{M}$.

 \begin{Theorem}[Two-sided upward]
	\label{thm:TwoSideUp}
	For $X_0=x \in [0,a]$ and $v\in (0,1]$, there exists a matrix $\bold{W}_v: \mathbb{N} \rightarrow \mathbb{R}^{N\times N}$ with $\bold{W}_v(0) = \bold{0}$ and $\bold{W}_1(\cdot) =: \bold{W}(\cdot)$, which is invertible such that
	\begin{linenomath*}\begin{equation}\label{twoexitprob}
			\mathbb{E}_x\left(v^{\tau^+_a};\tau^+_a < \tau^-_0, J_{\tau^+_a} \right) = \bold{W}_v(x)\bold{W}_v(a)^{-1},
	\end{equation}\end{linenomath*}
	where $\bold{W}_v(\cdot)$ satisfies
	\begin{linenomath*}\begin{equation}
			\label{eqWTran}
			\sum_{n=0}^\infty z^n \bold{W}_v(n) = \Bigl( v\bold{F}(z) - \bold{I}  \Bigr)^{-1},
	\end{equation}\end{linenomath*}
	for  $z \in (0, 1]$ and $z\notin \Gamma(\bold{G}_v)$. Additionally, we have the alternative representation
	\begin{linenomath*}\begin{equation}\label{wld}
			\bold{W}_v(n) = \bold{G}_v^{-n} \bold{L}^+_v(n),
	\end{equation}\end{linenomath*}
	where $\bold{L}^+_v(n) := \bold{L}_v(0, \tau^+_n -1)$
denotes the \textit{occupation time} at level 0 before hitting the upper level $n \in \mathbb{N}^+$ for a general, unrestricted MAC.
\end{Theorem}

\begin{Theorem}[Two-sided downward]
	\label{thm:joint}
	For $X_0=x \in [0,a]$, $v \in [0,1]$ and at least $z \in [0,1]$ such that $z \notin \Gamma\left(\bold{G}_v\right)$, we have

	\begin{equation}
		\label{eq:joint}
		\mathbb{E}_x\left(v^{\tau_0^-}z^{-X_{\tau_0^-}}; \tau_0^- < \tau_a^+, J_{\tau_0^-} \right) = z^{-1}\left[\bold{Z}_v(z, x-1) - \bold{W}_v(x)\bold{W}_v(a)^{-1}\bold{Z}_v(z, a-1)\right],
	\end{equation}
where
\begin{linenomath*}\begin{equation}
		\label{eqZMat}
		\bold{Z}_{v}(z,n) = z^{-n} \Bigl[\bold{I} + \sum_{k=0}^n z^k\, \bold{W}_{v}(k) \bigl(\bold{I} - {v}\bold{F}(z) \bigr)   \Bigr],
\end{equation}\end{linenomath*}
with $\bold{Z}_v(z, n) = z^{-n}\bold{I}$ for $n \leqs 0$ and $\bold{Z}_1(z,n) =: \bold{Z}(z, n)$, for all $z$.
\end{Theorem}


\noindent Note that a joint transform of this kind was not given in Theorem \ref{thm:TwoSideUp} since, by the upward skip-free property of the MAC, it follows that $X_{\tau^+_a} = a$ a.s..

Finally, we have a corollary for the discounted deficit below zero for the two-sided exit problem, which will play a vital role in the derivation of the dual dividend problem in Section \ref{sec:results}.

	\begin{Corollary}
		\label{cor:deficit}
		For $x \in (0,a]$, $v \in [0,1]$ and at least $z \in [0,1]$ such that $z \notin \Gamma\left(\bold{G}_v\right)$, we have
		\begin{align*}
			\mathbb{E}_x\left(v^{\tau_0^-}X_{\tau_0^-}; J_{\tau_0^-}, \tau_0^- < \tau_a^+ \right) &= \left[\bold{Z}_v(1, x-1) - \bold{W}_v(x)\bold{W}_v(a)^{-1}\bold{Z}_v(1, a-1)\right] \\
			& \hspace{20mm}- \left[\bold{Z}'_v(1, x-1) - \bold{W}_v(x)\bold{W}_v(a)^{-1}\bold{Z}'_v(1, a-1)\right] \\
			&= \bold{Y}_v(x-1) - \bold{W}_v(x)\bold{W}_v(a)^{-1}\bold{Y}_v(a-1),
		\end{align*}
		where $\bold{Z}'_v(1, x) = \frac{d}{dz}\bold{Z}_v(z, x)\big|_{z=1}$ and
		\begin{equation}
			\label{eq:Y}
			\bold{Y}_v(x) = \bold{Z}_v(1, x) - \bold{Z}'_v(1, x),
		\end{equation}
	with $\bold{Y}_v(0) = \bold{Z}_v(1, 0) = \bold{I}$.
	\end{Corollary}
	\begin{proof}
		To prove this result, we first note that
		\begin{equation*}
			\mathbb{E}_x\left(v^{\tau_0^-}X_{\tau_0^-}; J_{\tau_0^-}, \tau_0^- < \tau_a^+ \right) = - \frac{d}{dz}\mathbb{E}_x\left(v^{\tau_0^-}z^{-X_{\tau_0^-}}; J_{\tau_0^-}, \tau_0^- < \tau_a^+ \right) \big|_{z=1}.
		\end{equation*}
		Hence, the result follows by differentiating the right hand side of Eq.\,\eqref{eq:joint} of Theorem \ref{thm:joint}, setting $z=1$ and taking the negative of the resulting expression.
\end{proof}


\begin{Remark}
The above results rely on the identification of the $\bold{W}_v$ scale matrix, which can be obtained by inverting the transform given in Eq.\,\eqref{eqWTran}, using standard inversion techniques. However, it is also worth pointing out that in the discrete case, the $Z$-transform can also be `inverted' via coefficient matching. That is, by expanding the right-hand side of \eqref{eqWTran} in terms of an infinite (matrix) summation and matching coefficients, we obtain the $\bold{W}_v$ scale matrix function.
\end{Remark}

\subsection{Time Reversal}
\label{sec:time-rev}

It is well known within the literature of random walks that time-reversal and the corresponding `duality lemma' (see \cite{feller2008introduction} for details) give rise to a number of interesting distributional results. This idea can be easily extended to MACs, although extra care has to be taken regarding the phase transitions of $\{J_n\}_{n \in \mathbb{N}}$ in reversed time.

Let us define the so-called \textit{time-reversed} process by $(\widetilde{X}, \widetilde{J}) := \{(\widetilde{X}_n, \widetilde{J}_n)\}_{n \in \mathbb{N}}$ such that for a fixed $T \in \mathbb{N}$, we have
\begin{linenomath*}\begin{equation}
\widetilde{X}_n :=X_T - X_{T-n} \quad \text{and} \quad  \widetilde{J}_n := J_{T-n}.
\end{equation}\end{linenomath*}
Then, if we assume that $\{J_n\}_{n \in \mathbb{N}}$ has stationary initial distribution, i.e. $J_0 \overset{d}{=}\boldsymbol{\pi}$, it follows that $\{\widetilde{J}_n\}_{n \in \mathbb{N}}$ is again a homogenous Markov chain with transition probability matrix, denoted $\bold{\widetilde{P}}$, given by
\begin{linenomath*}\begin{equation}
		\bold{\widetilde{P} }= \text{diag}(\boldsymbol{\pi})^{-1}\bold{P^\top} \text{diag}(\boldsymbol{\pi}),
\end{equation}\end{linenomath*}
and the time-reversed process $(\widetilde{X}, \widetilde{J})$ is itself a MAC with probability generating matrix $\bold{\widetilde{F}}(z)$, given by
\begin{linenomath*}\begin{align*}\label{F}
\bold{\widetilde{F}}(z) = \sum_{m=-1}^\infty z^{m}\widetilde{\bold{A}}_{-m} &= \sum_{m=-1}^\infty z^{m}\text{diag}(\boldsymbol{\pi})^{-1}\bold{A}^\top_m \text{diag}(\boldsymbol{\pi}) \\
&=  \text{diag}(\boldsymbol{\pi})^{-1}\bold{F}(z)^\top \text{diag}(\boldsymbol{\pi}).
\end{align*}\end{linenomath*}

\begin{Proposition} \label{RMatrix}
\noindent Define $\widetilde{\bold{G}}_v$ to be the time-reversed counterpart of $\bold{G}_v$, such that
\begin{equation*}
	\widetilde{\bold{G}}_v = \mathbb{E}\left(v^{\widetilde{\tau}^+_1}; \widetilde{J}_{\widetilde{\tau}^+_1}\right),
\end{equation*}
where $\widetilde{\tau}^+_1 = \inf\{n \in \mathbb{N}: \widetilde{X}_n \geqslant 1\}$. Then
\begin{linenomath*}\begin{equation}
\bold{R}_v := \emph{diag}(\boldsymbol{\pi})^{-1}\widetilde{\bold{G}}_v^\top \emph{diag}(\boldsymbol{\pi}).
\end{equation}\end{linenomath*}
is the left solution of $\bold{F}(\cdot) = v^{-1}\bold{I}$.
\end{Proposition}
\begin{proof}
As $\widetilde{\bold{G}}_{v}$ is the time-reversed counterpart of $\bold{G}_v$ then, from Theorem \ref{ThmOneUp}, it is the right solution of $\widetilde{\bold{F}}(\cdot) = {v}^{-1}\bold{I}$ and satisfies
\begin{align*}
{v}^{-1}\bold{I} = \widetilde{\bold{F}}(\widetilde{\bold{G}}_{v}) =\sum_{m=-1}^\infty \widetilde{\bold{A}}_{-m}\widetilde{\bold{G}}^m_{v}
\end{align*}
Transposing both sides of this expression and multiplying on the left and right by $\text{diag}(\boldsymbol{\pi})^{-1}$ and $\text{diag}(\boldsymbol{\pi})$, respectively, gives
\begin{linenomath*}\begin{align*}
v^{-1}\bold{I} &= \text{diag}(\boldsymbol{\pi})^{-1}\left( \sum_{m=-1}^\infty \widetilde{\bold{A}}_{-m}\widetilde{\bold{G}}^m_{v}  \right)^\top \text{diag}(\boldsymbol{\pi}) \\
&= \text{diag}(\boldsymbol{\pi})^{-1} \sum_{m=-1}^\infty \left(\widetilde{\bold{G}}^\top_{v}\right)^m \widetilde{\bold{A}}^\top_{-m} \text{diag}(\boldsymbol{\pi}) \\
&= \text{diag}(\boldsymbol{\pi})^{-1} \sum_{m=-1}^\infty \left(\widetilde{\bold{G}}^\top_{v}\right)^m  \text{diag}(\boldsymbol{\pi})\,  \text{diag}(\boldsymbol{\pi})^{-1} \widetilde{\bold{A}}^\top_{-m} \text{diag}(\boldsymbol{\pi}) \\
&= \sum_{m=-1}^\infty \bold{R}_{v}^m \bold{A}_{-m},
\end{align*}\end{linenomath*}
which completes the proof.
\end{proof}

\noindent The matrix $\bold{R}_v$ is considered another of the fundamental matrices of a MAC, along with $\bold{G}_v$ and $\bold{L}_v$. In a similar way to $\bold{G}_v$, this matrix can only be obtained explicitly in some special cases but can be approximated using the numerical methods discussed in Remark \ref{rem:numsol}.  A probabilistic interpretation of the matrix $\bold{R}_1=:\bold{R}$, is given in \cite{AsmussenApplied}, where it is interpretted as a matrix with $(i,j)$-th element denoting the expected number of visits to the level $1$, whilst in phase $j \in E$, before first returning to the level $0$, given that $X_0=0$ and $J_0=i \in E$. In fact, using this definition and the Markov additive property, it is possible to prove Proposition \ref{RMatrix}, for $v=1$, directly. (see \cite{AsmussenApplied} and \cite{Ivanovs2019} for more details).


\begin{Remark} \label{rem:GR}
When reduced to the scalar case, the fundamental matrices $\bold{G}$ and $\bold{R}$ coincide and correspond to the smallest (positive) root of the (discrete) Lundberg equation, known in the literature as Lundberg's coefficient and has been studied in great detail (see \cite{Avram2019} and references therein). As a consequence, we observe that for the scalar random walk, based on the definitions of $\bold{G}$ and $\bold{R}$ given above, that the probability of hitting an upper level is equivalent to the expected number of visits to this level before returning to zero.
\end{Remark}

\section{Main Results}
\label{sec:results}

In this section, we present the main results for the G-S function and expected accumulated discounted dividends until ruin for the regular and dual risk processes defined in Eqs.\,\eqref{eqCDTRM1} and \eqref{eqDual}, respectively.


The crucial observation leading to the results of this section is that discrete-time risk process, $\{U_n\}_{n \in \mathbb{N}}$, paired with the influencing external Markov chain $\{J_n\}_{n \in \mathbb{N}}$, forms an upward skip-free  MAC, with initial value $U_0 = u \in \mathbb{N}$. To see this, note that the surplus process can alternatively be expressed as
\begin{equation*}
U_n = u + Y_1 +\cdots + Y_n,
\end{equation*}
where the variables $Y_k := 1 - Z_k$ for $k \in \mathbb{N}^+$, form a sequence of conditionally i.i.d.\, variables with distribution depending on $\{J_n\}_{n \in \mathbb{N}}$. This is identical to the form of a MAC as described in Eq.\,\eqref{eq:X}, with $U_0 = u \in \mathbb{N}$ and $\bold{A}_m = \bold{\Lambda}(1-m)$ for $m \leqs 1$. As such, it follows that the ruin times $\tau_0$ and $\widehat{\tau}_0$ for the regular and dual risk processes can be described in terms of downward and upward (by reflection arguments) crossing times of a MAC, respectively.

  We will now use these facts, along with the results given in Section \ref{MACs}, to derive semi-explicit expressions for the G-S function and expected accumulated discounted dividends until ruin for both risk process. In the following, it will be assumed that the natural net profit conditions, as stated in Eq.\,\eqref{eq:drift}, are in force. That is, for the regular risk process we assume that $\kappa'(1) < 0$ such that $U_n \rightarrow + \infty$ and, based on a reflection argument, for the dual risk model we assume $\kappa'(1) > 0$, such that $\widehat{U}_n \rightarrow + \infty$.

\subsection{Gerber-Shiu function - Regular risk process}

It is well known that for spectrally-negative L\'evy processes, the G-S function can be obtained in terms of so-called \textit{q-potential (resolvent) measures} and their corresponding densities, which describe the expected (discounted) time the L\'evy process spends at a given level (see \cite{kyprianou2013gerber} and references therein). Moreover, it has been shown in \cite{ivanovs2014potential}, that the potential measures also exist in the more general MAP framework and can be expressed in terms of the continuous-time occupation densities and scale matrices. In this section, we will show that the G-S function for a discrete-time MAC can be written in terms of so-called \textit{v-potential functions} and derive expressions for these in terms of scale matrices and other fundamental matrices of the MAC.

For $i,j \in E$, let us denote by $H^{(v)}_{ij}(x, \tau_0; u)$, the \textit{v-potential function} of the MAC $\{(U_n, J_n)\}_{n \in \mathbb{N}}$, with initial level $U_0 = u \in \mathbb{N}^+$, killed on exiting from the set of positive integers, which is defined by
\begin{linenomath*}\begin{align}\label{r-potential}
H^{(v)}_{ij}(x, \tau_0; u) &= \mathbb{E}_u \left[\sum_{n=0}^{\infty} v^{n} \mathbb{I}_{\{U_n = x, J_n = j, n < \tau_0 \}} \bigg| J_0 = i \right] \notag \\
& = \sum_{n=0}^{\infty} \mathbb{E}_u \left[ v^{n} \mathbb{I}_{(U_n = x, J_n = j, n < \tau_0)} | J_0 = i \right] \notag \\
& = \sum_{n=0}^{\infty} v^{n} \mathbb{P}_u \left( U_n = x, J_n = j, n < \tau_0 | J_0 = i \right).
\end{align}\end{linenomath*}

\noindent Then, by employing the law of total probability, the G-S function defined in Eq.\,\eqref{GSfunction}, can be expressed in terms of $H^{(v)}_{ij}(\cdot, \tau_0 ; u)$ as shown in the following proposition.

\begin{Proposition}\label{PropPot}
For $i,j\in E$, the G-S functions $\phi_{ij}(u)$, can be expressed as
\begin{equation}\label{mainformulaGerberShiufunction}
\phi_{ij}(u)= v\sum_{k=1}^\infty \sum_{m=k}^\infty
\sum_{l=1}^N w(k, k-m) H^{(v)}_{il}(k, \tau_0; u) \lambda_{lj}(m)
\end{equation}
where $w(\cdot, \cdot)$ is the penalty function and $\lambda_{ij}(\cdot)$ is the phase-dependent claim size distribution.
\end{Proposition}

\begin{proof}
Recalling the definition of the G-S functions $\phi_{ij}(u)$ from Eq.\,\eqref{GSfunction}, the law of total probability gives
\begin{linenomath*}\begin{align*}
\phi_{ij}(u) &= \sum_{n = 1}^\infty \sum_{k = 1}^\infty \sum_{m = k}^\infty \sum_{l = 1}^N v^{n}w(k, m-k) \\
&\hspace{10mm} \times \mathbb{P}_u\left(\tau_0 = n, U_{n-1} = k, U_n = k-m, J_n = j, J_{n-1} = l | J_0 = i \right) \\
&= \sum_{n = 1}^\infty \sum_{k = 1}^\infty \sum_{m = k}^\infty \sum_{l = 1}^N v^{n}w(k, m-k) \\
&\hspace{10mm} \times \mathbb{P}_u\left(\underline{U}_{n-1} > 0, U_{n-1} = k, J_{n-1} = l | J_0 = i \right)\mathbb{P}\left(Z_n = m, J_n = j |J_{n-1} = l \right) \\
&= \sum_{n = 1}^\infty \sum_{k = 1}^\infty \sum_{m = k}^\infty \sum_{l = 1}^N v^{n}w(k, m-k) \\
&\hspace{10mm} \times \mathbb{P}_u\left(\underline{U}_{n-1} > 0, U_{n-1} = k, J_{n-1} = l | J_0 = i \right)\lambda_{lj}(m)
\end{align*}\end{linenomath*}
where $\underline{U}_{n}:= \inf_{0 \leqs k \leqs n}\{ U_k\}$ and the second equality follows from the Markov property of $\{J_n\}_{n \in \mathbb{N}}$ along with the conditional i.i.d.\,property of the claim sizes. Then, after some re-arranging we obtain
\begin{linenomath*}\begin{align*}
\phi_{ij}(u) &=  v\sum_{k = 1}^\infty \sum_{m = k}^\infty \sum_{l = 1}^N w(k, m-k) \\
&\hspace{10mm} \times \sum_{n = 1}^\infty v^{(n-1)}\mathbb{P}_u\left(\underline{U}_{n-1} > 0, U_{n-1} = k, J_{n-1} = l |  J_0 = i \right)\lambda_{lj}(m),
\end{align*}\end{linenomath*}
from which the result follows by noting that the summation in the last line is equivalent to the $v$-potential function $H^{(v)}_{il}(k, \tau_0; u)$ defined in Eq.\,\eqref{r-potential}.
\end{proof}

\noindent The result of Proposition \ref{PropPot} provides a representation for the G-S function as long as we can identify the $v$-potential measures $H^{(v)}_{ij}(\cdot, \tau_0;u)$, for all $i,j \in E$. Using a similar methodology to \cite{ivanovs2014potential}, in the next theorem we show that the $v$-potential functions can actually be expressed in terms of the scale matrix $\bold{W}_v(\cdot)$ and fundamental matrices associated to the MAC $\{(U_n,J_n)\}_{n \in \mathbb{N}}$.

\begin{Theorem}
\label{potenetialthm}
Let, $\bold{H}^{(v)}(x, \tau_0;u)$ denote the $N$-dimensional  square matrix with $(i,j)$-th element given by the $v$-potential function $H^{(v)}_{ij}(x, \tau_0; u)$ for $i,j \in E$. Then, we have
\begin{equation*}
\bold{H}^{(v)}(x, \tau_0;u)=\bold{W}_{v}(u)\bold{R}_{v}^{x} -\bold{W}_{v}(u-x),
\end{equation*}
where $\bold{R}_{v}$ is the left solution of $\bold{F}(\cdot) = v^{-1}\bold{I}$.
\end{Theorem}

\begin{proof}
	
To begin, recall the definition of $\bold{L}_v^+(n)$ defined in Theorem \ref{thm:TwoSideUp} and let $\bold{L}_v := \bold{L}_v(0,\infty)$ denote the occupation time at the level 0 for an unrestricted MAC (with initial level $x=0$) over an infinite-time horizon. Then, by application of the strong Markov property and Markov additive property, it follows that
\begin{linenomath*}\begin{align}
\label{StrongMark}
\bold{L}_v &= \bold{L}^+_v(n) + \mathbb{E}\left( v^{\tau^+_n};J_{\tau^+_n}\right)\mathbb{E}\left( v^{\tau^{\{-n\}}};J_{\tau^{\{-n\}}}\right)\bold{L}_v \notag  \\
& =  \bold{L}^+_v(n) + \bold{G}^n_v \, \mathbb{E}\left( v^{\tau^{\{-n\}}};J_{\tau^{\{-n\}}}\right)\bold{L}_v,
\end{align}\end{linenomath*}
where $\tau^{\{k\}}$ denotes the hitting time defined in Eq.\eqref{eq:hitting}. Re-arranging the above expression gives
\begin{linenomath*}\begin{align*}
\label{keyequation}
\mathbb{E}\left( v^{\tau^{\{-n\}}};J_{\tau^{\{-n\}}}\right) &= \bold{G}^{-n}_v\left[\bold{L}_v - \bold{L}^+_v(n))\right]\bold{L}_v^{-1} \notag \\
&= \bold{G}^{-n}_v - \bold{G}^{-n}_v\bold{L}^+_v(n)\bold{L}_v^{-1} \notag \\
&= \bold{G}^{-n}_v - \bold{W}_v(n)\bold{L}_v^{-1},
\end{align*}\end{linenomath*}
where, in the last equality, we have used the form of the scale matrix given in Eq.\,\eqref{wld} of Theorem \ref{thm:TwoSideUp} and that fact that $\bold{L}_v^{-1}$ exists since we assume that $\bold{A}_{1}$ is invertible (see Remark 5 of \cite{ourpaper}). The above identity, along with Eq.\,\eqref{GvMatrix}, shows that for any $n \in \mathbb{Z}$, we have
\begin{linenomath*}\begin{equation}\label{newidentity}
\mathbb{E}\left(v^{\tau^{\{n\}}}\right)=\bold{G}_{v}^{n} -\bold{W}_v(-n)\bold{L}_v^{-1},
\end{equation}\end{linenomath*}
where we have used the fact that $\bold{W}_v(-n)=\bold{0}$ for $n \in  \mathbb{N}$.


Now, by recalling the definition of the $v$-potential function defined in Eq.\,\eqref{r-potential}, we can apply a similar idea to that of Eq.\,\eqref{StrongMark}, to show that for $k, u \in \mathbb{N}^+$
\begin{linenomath*}\begin{align}
\bold{H}^{(v)}(k,\tau_0;u) &= \mathbb{E}\left(v^{\tau^{\{k-u\}}}; J_{\tau^{\{k-u\}}}\right) \bold{L}_v -\mathbb{E}\left(v^{\tau^{\{-u\}}}; J_{\tau^{\{-u\}}}\right)\mathbb{E}\left(v^{\tau^{\{k\}}}; J_{\tau^{\{k\}}}\right)\bold{L}_v  \notag \\
& = \left[\mathbb{E}\left(v^{\tau^{\{k-u\}}}; J_{\tau^{\{k-u\}}}\right) -\mathbb{E}\left(v^{\tau^{\{-u\}}}; J_{\tau^{\{-u\}}}\right) \bold{G}^k_v\right] \bold{L}_v .
\end{align}\end{linenomath*}
Finally, by substituting the identity in Eq.\,\eqref{newidentity} into the right hand side of the above expression, we have
\begin{linenomath*}\begin{equation*}
\bold{H}^{(v)}(k,\tau_0;u) =  \bold{W}_{v}(u) \bold{L}_{v}^{-1}\bold{G}_{v}\bold{L}_{v} - \bold{W}_{v}(u-k).
\end{equation*}\end{linenomath*}
and the proof is complete once we show that
\begin{linenomath*}\begin{equation}\label{GR}
\bold{L}_{v}^{-1}\bold{G}_{v}\bold{L}_{v} = \bold{R}_{v}.
\end{equation}\end{linenomath*}
To do this, first observe that by the strong Markov property, it follows that  $\bold{L}_v(1,\infty) = \bold{G}_v\bold{L}_v$  and thus
\begin{linenomath*}\begin{align}\label{3}
		\bold{L}_v^{-1} \bold{G}_{v}\bold{L}_v =\bold{L}_v^{-1} \bold{L}_v\left(1, \infty \right).
\end{align}\end{linenomath*}
Now, if we define $\widetilde{\bold{L}}_v\left(x, \infty \right)$ to be the time-reversed counterpart of $\bold{L}_v\left(x, \infty \right)$, then from Theorem \ref{thm:Lv} it is easy to see that
\begin{equation*}
\bold{L}_v(x, \infty) = \text{diag}(\boldsymbol{\pi})^{-1} \widetilde{\bold{L}}_v\left(x, \infty \right)^\top  \text{diag}(\boldsymbol{\pi})
\end{equation*}
for all $x\in \mathbb{Z}$ and thus
\begin{linenomath*}\begin{align}\label{1}
\bold{L}_v\left(1, \infty \right)&= \text{diag}(\boldsymbol{\pi})^{-1} \widetilde{\bold{L}}_v\left(1, \infty \right)^\top  \text{diag}(\boldsymbol{\pi}) \notag \\
&= \text{diag}(\boldsymbol{\pi})^{-1} \left(\widetilde{\bold{G}}_{v}\widetilde{\bold{L}}_v \right)^\top  \text{diag}(\boldsymbol{\pi}) \notag \\
&= \text{diag}(\boldsymbol{\pi})^{-1} \widetilde{\bold{L}}_v ^\top \widetilde{\bold{G}}_{v}^\top  \text{diag}(\boldsymbol{\pi}).
\end{align}\end{linenomath*}
Similarly, it follows that $\bold{L}_v = \text{diag}(\boldsymbol{\pi})^{-1} \widetilde{\bold{L}}_v^\top  \text{diag}(\boldsymbol{\pi})$, such that
\begin{linenomath*}\begin{equation}\label{2}
\bold{L}_v^{-1} = \text{diag}(\boldsymbol{\pi})^{-1} \left(\widetilde{\bold{L}}_v^\top\right)^{-1}  \text{diag}(\boldsymbol{\pi}).
\end{equation}\end{linenomath*}
Finally, combining Eq's.\,\eqref{3}, \eqref{1} and \eqref{2}, yields
\begin{linenomath*}\begin{equation*}
\bold{L}_v^{-1} \bold{G}_{v}\bold{L}_v  = \text{diag}(\boldsymbol{\pi})^{-1} \widetilde{\bold{G}}_v^\top  \text{diag}(\boldsymbol{\pi}) = \bold{R}_v,
\end{equation*}\end{linenomath*}
which, by Proposition \ref{RMatrix} is the left solution to the equation $\bold{F}(\cdot) = {v}^{-1}\bold{I}$, as required.
\end{proof}

\noindent Combining the results from Proposition \ref{PropPot} and Theorem \ref{potenetialthm} above, leads to an expression for the G-S function in terms of the scale matrices and is presented in the following theorem.

\begin{Theorem}
For $u \in \mathbb{N}$, the G-S function $\phi(u)$ is given by
\begin{linenomath*}\begin{equation}\label{GSMAC1}
\phi(u) = \boldsymbol{\pi}^\top \bold{\Phi}(u) \boldsymbol{e},
\end{equation}\end{linenomath*}
where

\begin{linenomath*}\begin{equation}\label{GSMAC2}
\bold{\Phi}(u) = v\sum_{k=1}^\infty \sum_{m = k}^\infty w(k,k-m)\left( \bold{W}_{v}(u)\bold{R}_{v}^{k} -\bold{W}_{v}(u-k) \right) \bold{\Lambda}(m).
\end{equation}\end{linenomath*}
\end{Theorem}

\begin{Corollary}
For $u \in \mathbb{N}$, the probability of ruin $\psi(u)$ is given by
\begin{linenomath*}\begin{equation}\label{RuinMAC}
\psi(u) = \boldsymbol{\pi}^\top \bold{\Psi}(u) \boldsymbol{e},
\end{equation}\end{linenomath*}
with
\begin{linenomath*}\begin{equation}\label{RuinMAC1}
\bold{\Psi}(u) = \sum_{k=1}^\infty \sum_{m = k}^\infty \left( \bold{W}(u)\bold{R}^{k} -\bold{W}(u-k) \right) \bold{\Lambda}(m).
\end{equation}\end{linenomath*}
\end{Corollary}

\subsection{Gerber-Shiu - Dual Risk Process} \label{GSDiv}

Recall that the dual risk process $\{\widehat{U}_n\}_{n \in \mathbb{N}}$ with initial value $u \in \mathbb{N}$, has dynamics (jumps) which are equivalent in distribution to a reflection of the regular process in the time axis, i.e., $\{\Delta \widehat{U}_n \}_{n \in \mathbb{N}} \overset{d}{=} \{- \Delta U_n \}_{n \in \mathbb{N}}$, whilst the phase process $\{J_n\}_{n \in \mathbb{N}}$ remains unchanged. In other words, the dual risk process is equivalent to a `spectrally positive' (downward-skip free) MAC. As such, it follows that exit times for the dual process coincide with corresponding exit times for the regular risk process. In particular, the ruin time for the dual risk process with initial capital $u \in \mathbb{N}$, denoted $\widehat{\tau}_0$, is equivalent (by reflection) to the hitting time of the level 0 for a spectrally negative MAC with initial value $-u$. Moreover, due to the translation invariance property of MACs, by shifting the level process this is also equivalent to the hitting time of the level $u \in \mathbb{N}$ of a spectrally negative MAC with initial level $0$, i.e., $\widehat{\tau}_0 \equiv \tau^+_u$ (see Fig:\ref{fig:1}).

\begin{figure}[h!]
\centering
\begin{subfigure}{.5\textwidth}
  \centering
  \includegraphics[width=8.0cm]{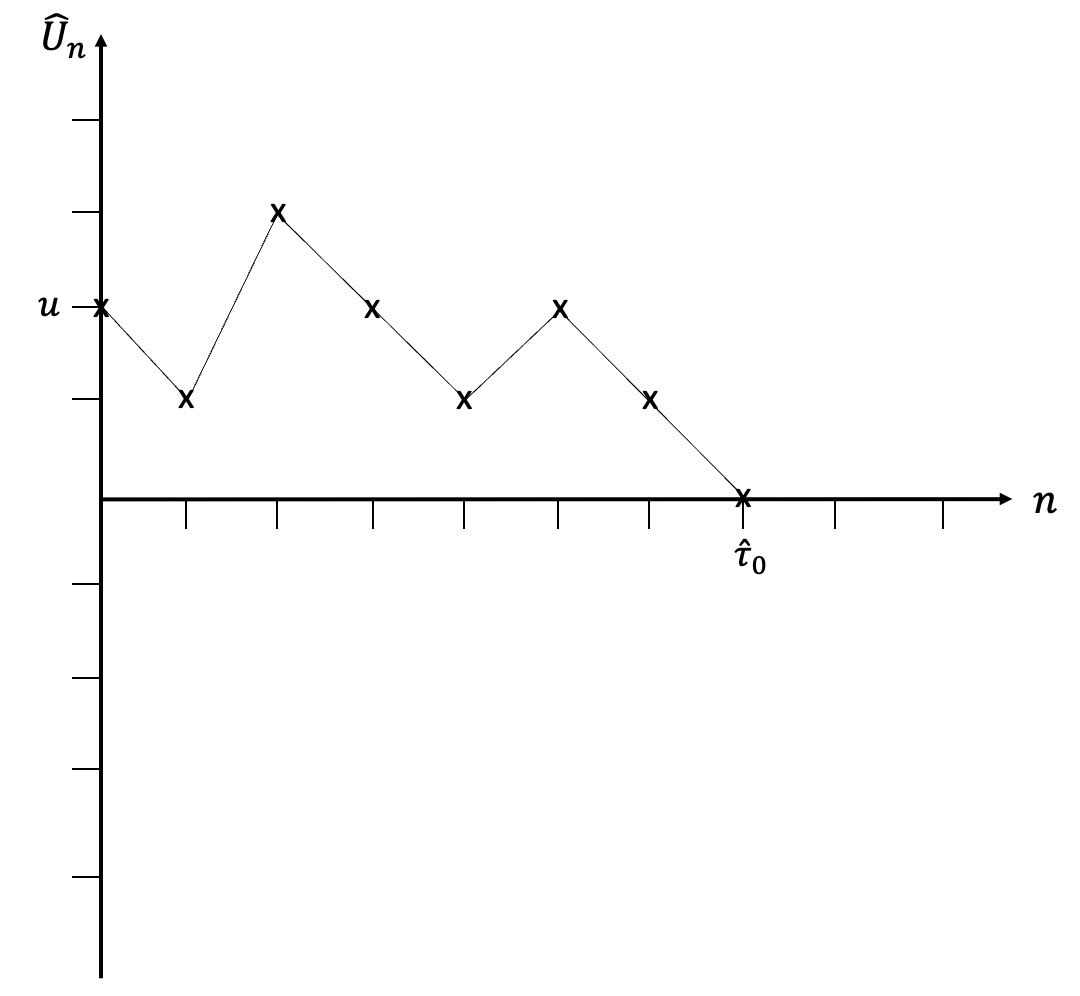}
  \label{fig:sub1}
\end{subfigure}%
\begin{subfigure}{.5\textwidth}
  \centering
  \includegraphics[width=8.0cm]{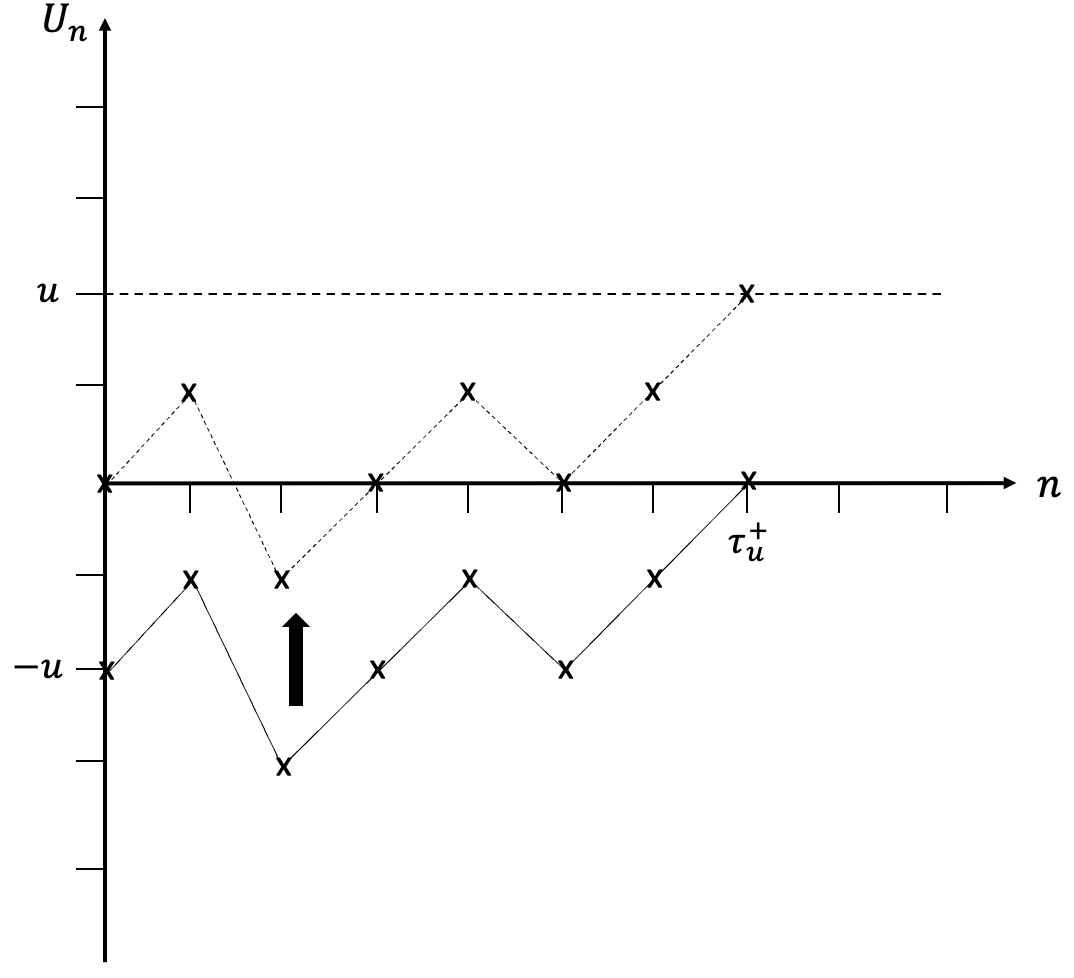}
  \label{fig:sub2}
\end{subfigure}
\caption{Equivalence of exit times between the dual and regular process. }
\label{fig:1}
\end{figure}

\noindent The above observation results directly in the following theorem.

\begin{Theorem} \label{G-SDual} For $u \in \mathbb{N}$, the G-S function for the dual risk process, namely $\widehat{\phi}(u)$, is given by
\begin{linenomath*}\begin{equation}
\widehat{\phi}(u) = \boldsymbol{\alpha}^\top \bold{G}_{v}^u \boldsymbol{e},
\end{equation}\end{linenomath*}
where $\boldsymbol{\alpha}^\top$ is the initial distribution of the phase process $\{J_n\}_{n \in \mathbb{N}}$ and the matrix $\bold{G}_v$ is defined in Eq.\,\eqref{GvMatrixOrig}.
\end{Theorem}
\begin{proof} From Figure: \ref{fig:1} and the preceding discussion, it is clear that the ruin time of the dual risk process is equivalent to the upward hitting time of the level $u \in \mathbb{N}$ for a spectrally negative MAC with initial value $X_0 = 0$, i.e., $\widehat{\tau}_0 \equiv \tau^+_u$. Hence, it follows that
\begin{linenomath*}\begin{align*}
\widehat{\phi}_{ij}(u) &= \mathbb{E}_u\left( v^{\widehat{\tau}_0}\mathbb{I}_{(\widehat{\tau}_0 < \infty,\, J_{\widehat{\tau}_0} = j )} \bigg|J_0 = i \right) \notag \\
&= \mathbb{E}\left( v^{\tau^+_u}\mathbb{I}_{ (\tau^+_u < \infty,\, J_{\tau^+_u} = j )} \bigg| J_0 = i \right) \notag \\
&= \left(\bold{G}^u_{v} \right)_{i,j \in E},
\end{align*}\end{linenomath*}
which, along with the definition of the unconditional G-S function given in Eq.\,\eqref{DGS}, gives the result.
\end{proof}

\begin{Corollary} For $u \in \mathbb{N}$, the ruin probability for the dual risk process, namely $\widehat{\psi}(u)$, is given by
\begin{linenomath*}\begin{equation}
\widehat{\psi}(u) = \boldsymbol{\alpha}^\top \bold{G}^u \boldsymbol{e},
\end{equation}\end{linenomath*}
where $\boldsymbol{\alpha}^\top$ is the initial distribution of the phase process $\{J_n\}_{n \in \mathbb{N}}$ and $\bold{G} := \bold{G}_1$.
\end{Corollary}

\begin{proof} The result follows directly by setting $v = 1$ in Theorem \ref{G-SDual}.
\end{proof}

\begin{Remark}
These results for the dual risk process have recently been derived in \cite{kim2022parisian} (Theorem 1), using similar conditioning arguments as those presented in previous sections.
\end{Remark}

\subsection{Constant Dividend Barrier Problem}

\noindent As discussed above, G-S theory covers a range of ruin and risk related measures in addition to those implicitly contained with the G-S function itself. One such quantity of interest is the expected accumulated discounted dividends until ruin under a (constant) dividend barrier strategy, where any surplus above the so-called dividend barrier $b \in \mathbb{N}$, is paid out as dividends to the shareholders.

With this in mind, let us introduce the amended risk process $\{V_n\}_{n \in \mathbb{N}}$, defined by
\begin{linenomath*}\begin{equation}
		\label{regulatedriskprocess}
		V_n=U_n-D_n
\end{equation}\end{linenomath*}
where $\{U_n\}_{n \in \mathbb{N}}$ denotes the regular risk process defined in Eq.\,\eqref{eqCDTRM1} and the `regulator' process
\begin{linenomath*}\begin{equation}
		\label{eqregulator}
		D_n = \left( \bar{U}_n \lor b\right) - b,
\end{equation}\end{linenomath*}
with $\bar{U}_n := \sup_{k \leqs n} U_k$, denotes the accumulated dividend payments up to period $n \in \mathbb{N}$ under a constant dividend barrier strategy with fixed dividend barrier $b\geqs u$. In a similar way, we can define the dividend regulated version of the dual risk process by $\{\widehat{V}_n\}_{n \in \mathbb{N}}$, such that
\begin{linenomath*}\begin{equation}\label{regulatedriskprocess}
		\widehat{V}_n=\widehat{U}_n-\widehat{D}_n
\end{equation}\end{linenomath*}
where $\{\widehat{D}_n\}_{n \in \mathbb{N}}$ denotes the accumulated dividends paid up to period $n \in \mathbb{N}$ under a dual risk model. Note that due to the upward-skip free property of the MAC, for the regular risk process dividends can only be paid at a unit rate per period whereas in the dual process, due to the presence of upward jumps, the dividend payments can take arbitrary size. Then, the expected accumulated discounted dividends until ruin, denoted by $v(u)$ and $\widehat{v}(u)$, for the regular and dual risk processes, are given by
\begin{linenomath*}\begin{equation*}
		\chi(u) = \boldsymbol{\alpha}^\top \bold{V}_v(u) \boldsymbol{e} \quad \text{and} \quad \widehat{\chi}(u) = \boldsymbol{\alpha}^\top \widehat{\bold{V}}_v(u) \boldsymbol{e},
\end{equation*}\end{linenomath*}
respectively, where
\begin{linenomath*}\begin{equation}\label{valuefunction}
		\left(\bold{V}_v(u)\right)_{i,j \in E}=\mathbb{E}_u\left[
		\sum_{k=1}^{\tau^b_0} v^{k} \Delta D_k \, \mathbb{I}_{(\tau^b_0 < \infty, J_{\tau^b_0} = j )} \bigg|J_0=i
		\right],
\end{equation}\end{linenomath*}
and
\begin{linenomath*}\begin{equation}\label{valuefunctiondual}
		\left(\widehat{\bold{V}}_v(u)\right)_{i,j \in E}= \mathbb{E}_u\left[
		\sum_{k=1}^{\widehat{\tau}^{\,b}_0} v^{k} \Delta \widehat{D}_k \, \mathbb{I}_{(\widehat{\tau}^{\,b}_0 < \infty, J_{\widehat{\tau}^{\,b}_0} = j )} \bigg|J_0=i
		\right],
\end{equation}\end{linenomath*}
with
\begin{linenomath*}\begin{equation}\label{ruintimereg}
		\tau^b_0=\inf\{n \in \mathbb{N}: V_n \leqs 0\}, \qquad \widehat{\tau}^{\,b}_0=\inf\{n \in \mathbb{N}: \widehat{V}_n \leqs 0\}.
\end{equation}\end{linenomath*}


\begin{Theorem}[Regular risk process]
\label{Thm:Divs}
For $v \in (0,1]$, it follows that $\bold{V}_v(b+x) = x + \bold{V}_v(b)$ for $x \in \mathbb{N}$, whilst for $u \in (0,b]$, we have
\begin{equation}
\label{eq:Div1}
\bold{V}_v(u) = \bold{W}_v(u)\left[\bold{W}_v(b+1) - \bold{W}_v(b)\right]^{-1}.
\end{equation}	
\end{Theorem}

\begin{proof}
The first result follows directly from the fact that any initial capital above the level $b$ will be paid out immediately (at time 0) as dividends. For the second result, based on the definition of the discounted dividends $\bold{V}_v(u)$ given in Eq.\,\eqref{valuefunction}, conditioning on the event $\{\tau^+_b < \tau^-_0\}$ and employing the strong Markov and Markov additive properties give
\begin{align}
	\label{eq:V(u)}
	\bold{V}_v(u) &= \mathbb{E}_u\left(v^{\tau^+_b};\tau^+_b < \tau^-_0, J_{\tau^+_b} \right)\bold{V}_v(b) \notag \\
	&=\bold{W}_v(u)\bold{W}_v(b)^{-1}\bold{V}_v(b),
\end{align}
where the second equality follows from the result of Theorem \ref{thm:TwoSideUp}. As such, it remains only to determine the boundary condition $\bold{V}_v(b)$.

Conditioning on the events in the first period of time, recalling that the matrix $\bold{A}_m$ denotes the probability transition matrix for the surplus process increasing by $m$ levels and noting that if the surplus increases from level $b$ to $b+1$ via `drift', then a unit dividend is paid out immediately and the surplus returns to the level $b$, it follows that
\begin{align*}
	\bold{V}_v(b) &= v\left[\bold{A}_1 \left(\bold{I} + \bold{V}_v(b) \right) + \sum_{k=0}^{b-1} \bold{A}_{-k}\bold{V}_v(b-k)\right] \\
	&= v\left[\bold{A}_1 \left(\bold{I} + \bold{V}_v(b) \right) + \sum_{k=0}^{b-1}\bold{A}_{-k}\bold{W}_v(b-k)\bold{W}_v(b)^{-1}\bold{V}_v(b)\right],
\end{align*}
where we have used Eq.\,\eqref{eq:V(u)} in the second equality. Re-arranging the above gives an equivalent expression of the form

\begin{equation}
	\label{eq:prerecur}
\left[ \left(\bold{W}_v(b) - v\sum_{k=0}^{b-1}\bold{A}_{-k}\bold{W}_v(b-k)\right)\bold{W}_v(b)^{-1} - v\bold{A}_1\right]\bold{V}_v(b) = v\bold{A}_1,
\end{equation}
which can be reduced further due to the recursive relationship between the $\bold{W}_v$ scale matrices stated in the following Lemma with proof given in the Appendix.

\begin{Lemma}
	\label{Lem:Recursion}
	Assume that $\bold{A}_{1}$ is invertible. Then, for $v \in (0,1]$ and $b \in \mathbb{N}^+$ the scale matrices $\bold{W}_v(x)$ satisfy the recursive relationship
	\begin{equation*}	
		v\bold{A}_{1}\bold{W}_v(b+1) = \bold{W}_v(b) - v \sum_{k=0}^{b-1}\bold{A}_{-k}\bold{W}_v(b-k).
	\end{equation*}
\end{Lemma}

\noindent From Lemma \ref{Lem:Recursion}, Eq.\eqref{eq:prerecur}, reduces to
\begin{equation*}
\left[ v\bold{A}_1\bold{W}_v(b+1)\bold{W}_v(b)^{-1} - v\bold{A}_1\right]\bold{V}_v(b) = v\bold{A}_1,
\end{equation*}
or equivalently
\begin{equation*}
	\left[ \bold{W}_v(b+1)\bold{W}_v(b)^{-1} - \bold{I}\right]\bold{V}_v(b) = \bold{I},
\end{equation*}
since it is assumed that $\bold{A}_{1}$ is invertible. Hence, the matrix $\left[ \bold{W}_v(b+1)\bold{W}_v(b)^{-1} - \bold{I}\right]$ is invertible and it follows that
\begin{equation*}
	\bold{V}(b) = \left[ \bold{W}_v(b+1)\bold{W}_v(b)^{-1} - \bold{I}\right]^{-1}.
\end{equation*}
Finally, substituting this form for $\bold{V}(b)$ back into Eq.\,\eqref{eq:V(u)}, we obtain
\begin{align*}
	\bold{V}(u) &= \bold{W}_v(u)\bold{W}_v(b)^{-1}\left[ \bold{W}_v(b+1)\bold{W}_v(b)^{-1} - \bold{I}\right]^{-1} \\
	&= \bold{W}_v(u)\left[ \bold{W}_v(b+1) - \bold{W}_v(b)\right]^{-1}.
\end{align*}

\end{proof}

\begin{Remark}
A vectorised version of the result of Theorem \ref{Thm:Divs} was first derived in \cite{Chen2014-hf} where only the initial phase of the external Markov chain was considered. Although it is not explicitly named in the paper, the proof of this result in \cite{Chen2014-hf} relies on an `auxiliary function' $W(\cdot)$ which is equivalent to the $\bold{W}_v$ scale matrix presented here.
\end{Remark}

\begin{Theorem}[Dual risk process]
\label{Thm:DivsDual}
For $v \in (0,1]$, it follows that $\bold{\widehat{V}}_v(b+x) = x + \bold{\widehat{V}}_v(b)$ for $x \in \mathbb{N}$, whilst for $u \in (0,b]$, we have
	\begin{align}
		\label{eq:DivDual}
		\bold{\widehat{V}}_v(u) &= \bold{Y}_v(b-u-1) - \bold{W}_v(b-u)\bold{W}_v(b)^{-1}\bold{Y}_v(b-1) \notag \\
		&\hspace{20mm}+ \left[\bold{Z}_v(1, b-u-1)-\bold{W}_v(b-u)\bold{W}_v(b)^{-1}\bold{Z}_v(1, b-1) \right]\bold{\widehat{V}}_v(b),
	\end{align}
where
\begin{align*}
	\bold{\widehat{V}}_v(b) &=  \left[\bold{I} - v\left\{\bold{A}_1\left(\bold{I} - \bold{W}_v(1)\bold{W}_v(b)^{-1}\bold{Z}_v(1, b-1)\right) + \sum_{k=0}^\infty \bold{A}_{-k}\right\} \right]^{-1} \\
	&\hspace{20mm}\times v\left\{\bold{A}_1\left(\bold{I} - \bold{W}_v(1)\bold{W}_v(b)^{-1}\bold{Y}_v(b-1)\right) + \sum_{k=0}^\infty k\bold{A}_{-k}\right\} .
\end{align*}
\end{Theorem}
\begin{proof}
The first part of the result is similar to that of Theorem \ref{Thm:Divs} and follows from the fact that any initial capital above the dividend level $b$ is paid out immediately as dividends.

For the second part of the result, recall that the dynamics of the dual model are equivalent in distribution to the reflection of the classic upward skip-free process. As such, by reflecting the dual process in the $x$-axis, shifting the resulting process upwards by $b$ and  employing the strong Markov property, it follows that
\begin{align}
	\label{eq:dualdiv1}
\bold{\widehat{V}}_v(u) &= \mathbb{E}_{b-u}\left(v^{\tau^-_0} X_{\tau^-_0} ;J_{\tau^-_0}, \tau^-_0 < \tau^+_b\right) + \mathbb{E}_{b-u}\left(v^{\tau^-_0} ;J_{\tau^-_0}, \tau^-_0 < \tau^+_b\right)\bold{\widehat{V}}_v(b).
\end{align}
The result of Eq.\,\eqref{eq:DivDual} follows directly by employing the results of Theorem \ref{thm:joint} and Corollary \ref{cor:deficit} of Section \ref{sec:ExitProb}.

For $\bold{\widehat{V}}_v(b)$, by conditioning on the events in the next period of time and applying the Markov additive property, we obtain
\begin{align*}
\bold{\widehat{V}}_v(b) &= v\bold{\widehat{A}}_{-1} \bold{\widehat{V}}_v(b-1) + v \sum_{k=0}^\infty \bold{\widehat{A}}_k\left(k + \bold{\widehat{V}}_v(b)\right) \\
&=  v\bold{A}_{1} \bold{\widehat{V}}_v(b-1) + v \sum_{k=0}^\infty k \bold{A}_{-k} + v\sum_{k=0}^\infty \bold{A}_{-k}\bold{\widehat{V}}_v(b).
\end{align*}
Substituting the form of Eq.\,\eqref{eq:DivDual} into the first term of the above expression and recalling that $\bold{Y}_v(1) = \bold{Z}_v(1,0) = \bold{I}$, after some re-arranging we obtain
\begin{align}
	\label{eq:dualgamma}
	&\left[\bold{I} - v\left\{\bold{A}_1\left(\bold{I} - \bold{W}_v(1)\bold{W}_v(b)^{-1}\bold{Z}_v(1, b-1)\right) + \sum_{k=0}^\infty \bold{A}_{-k}\right\} \right]\bold{\widehat{V}}_v(b) \notag \\
	&\hspace{50mm}=  v\left\{\bold{A}_1\left(\bold{I} - \bold{W}_v(1)\bold{W}_v(b)^{-1}\bold{Y}_v(b-1)\right) + \sum_{k=0}^\infty k\bold{A}_{-k}\right\} .
\end{align}
Finally, the result follows after multiplication (on the left) by $$\left[\bold{I} - v\left\{\bold{A}_1\left(\bold{I} - \bold{W}_v(1)\bold{W}_v(b)^{-1}\bold{Z}_v(1, b-1)\right) + \sum_{k=0}^\infty \bold{A}_{-k}\right\} \right]^{-1},$$ on both sides of Eq.\,\eqref{eq:dualgamma}.  The existence of this matrix follows from diagonal dominance. To see this, note that
the entries of the matrix $\left(\bold{I} - \bold{W}_v(1)\bold{W}_v(b)^{-1}\bold{Z}_v(1, b-1)\right)$ 
are non-negative and less than one since, by \eqref{eq:joint}, it is equivalent to 
$\mathbb{E}_1\left(v^{\tau^-_0}  ;J_{\tau^-_0}, \tau^-_0 < \tau^+_b\right)$.
Moreover,  we have $\bold{A}_1+\sum_{k=0}^\infty \bold{A}_{-k}=\bold{P}$.
Hence, the sum of the entries in each row of the matrix  
$v\left\{\bold{A}_1\left(\bold{I} - \bold{W}_v(1)\bold{W}_v(b)^{-1}\bold{Y}_v(b-1)\right) + \sum_{k=0}^\infty \bold{A}_{-k}\right\}$ are 
strictly less than one. This completes the proof. 
\end{proof}


\section*{Appendix}

\begin{proof}[\bf Proof of Lemma \ref{Lem:Recursion}]
For $a, b \in \mathbb{N}^+$, by conditioning on the events in the first period of time and applying the Markov additive property, it follows that
\begin{align*}
	\mathbb{E}\left(v^{\tau^+_a};J_{\tau^+_a}, \tau^+_a < \tau^-_{-b} \right) &= v \sum_{k=-1}^{b-1}\bold{A}_{-k}\mathbb{E}_{-k}\left(v^{\tau^+_a};J_{\tau^+_a}, \tau^+_a < \tau^-_{-b} \right) \\
	&= v \sum_{k=-1}^{b-1}\bold{A}_{-k}\mathbb{E}\left(v^{\tau^+_{a+k}};J_{\tau^+_{a+k}}, \tau^+_{a+k} < \tau^-_{-(b-k)} \right),
\end{align*}
which, from Theorem \ref{thm:TwoSideUp} and the assumption that $\bold{A}_1$ is invertible, is equivalent to
\begin{equation*}
	\bold{W}_v(b)\bold{W}_v(a+b)^{-1} = v \sum_{k=-1}^{b-1}\bold{A}_{-k} \bold{W}_v(b-k)\bold{W}_v(a+b)^{-1}.
\end{equation*}
The result follows after multiplying through on the right by $\bold{W}_v(a+b)$ and re-arranging.
\end{proof}

\newpage
\bibliography{DTGS-arXiv}

\end{document}